\documentclass[12pt]{article}

\usepackage{amsthm,amsfonts,amsbsy,amssymb,amsmath}
\usepackage[all]{xy}
\usepackage[colorlinks=true]{hyperref}

\DeclareMathOperator{\charr}{char}

\begin{document}
\title{Relative Noether inequality on fibered surfaces}
\author{Xinyi Yuan
\thanks{Department of Mathematics, University of California, Berkeley, CA 94720, U.S.A.
Email: yxy@math.berkeley.edu}
\and
Tong Zhang
\thanks{Department of Mathematics, University of Alberta, Edmonton, AB T6G 2G1, Canada
Email: tzhang5@ualberta.ca}
}
\date{\today}


\maketitle

\theoremstyle{plain}
\newtheorem{theorem}{Theorem}[section]
\newtheorem{lemma}[theorem]{Lemma}
\newtheorem{coro}[theorem]{Corollary}
\newtheorem{prop}[theorem]{Proposition}
\newtheorem{defi}[theorem]{Definition}

\theoremstyle{remark}\newtheorem{remark}[theorem]{Remark}

\tableofcontents

\section{Introduction}

This paper is the algebraic version of the previous work \cite{YZ} of the authors on linear series on arithmetic surfaces.

We prove effective upper bounds on the global sections of nef line bundles of small generic degree over a fibered surface over a field of any characteristic.
It can be viewed as a relative version of the classical Noether inequality for surfaces.

As a consequence, we give a new proof of the slope inequality for fibered surface without using any
stability method. The treatment is essentially different from these of Xiao, Cornalba--Harris and Moriwaki. 
We also study the geography problem of surfaces in positive characteristics and show that the Severi
inequality is true for surfaces of general type in positive characteristic whose Albanese map is generically finite.
Moreover, the geography of surfaces with Albanese fibrations is studied.

We would like to point out that most results in this paper, except the slope inequality, are new in positive characteristic. Nevertheless, we will state our
results in full generality, since in characteristic $0$ they still hold and Theorem \ref{relnoether} and \ref{relnoetheromega} have not been stated yet in previous literatures.

\subsection{Relative Noether inequality}

Let $k$ be an algebraically closed field of any characteristic.
Let $f: X \to Y$ be a \textit{surface fibration of genus $g$} over $k$. That is:
\begin{itemize}
\item [(1)] $X$ is a smooth projective surface over $k$;
\item [(2)] $Y$ is a smooth projective curve over $k$;
\item [(3)] $f$ is flat and the general fiber $F$ of $f$ is a geometrically integral curve of arithmetic genus $g:=p_a(F)$.
\end{itemize}

Note that here we do \textit{not} assume the general fiber $F$ to be smooth, since it is not always
true if $\charr k>0$. See \cite{Li1}.

A reduced and irreducible curve $C$ over $k$ is called \textit{hyperelliptic} if its arithmetic genus
$p_a(C) \ge 2$ and if there exists a flat morphism of degree $2$ from $C$ onto $\mathbb P^1_k$.
It follows that $C$ is automatically Gorenstein (cf. \cite{Li2}). If $F$ is hyperelliptic,
then $f$ is called a \textit{hyperelliptic fibration}. Otherwise, $f$ is called
a \textit{non-hyperelliptic fibration}.

The following is the main theorem of this paper.

\begin{theorem} \label{relnoether}
Let $f: X \to Y$ be a  surface fibration of genus $g \ge 2$ over $k$,
and $L$ be a nef line bundle on $X$. Denote $d=\deg(L|_F)$, where $F$ is a general fiber
of $f$. If $2 \le d \le 2g-2$, then
$$
h^0(L) \le (\frac 14 + \frac {2+\varepsilon}{4d}) L^2 + \frac {d+2-\varepsilon}{2}.
$$
Here $\varepsilon=1$ if $F$ is hyperelliptic, $d$ is odd and $L|_F \le K_F$. Otherwise, $\varepsilon=0$.
\end{theorem}

Note that this inequality is sharp and $\varepsilon$ is necessary. For example, let $f$ be a trivial hyperelliptic fibration
and $L$ be the pullback of a divisor $D$ on $F$ satisfying $2h^0(D)=\deg D + 1$. In this case, we obtain the equality with $\varepsilon=1$.

The most interesting case of the theorem occurs when $L$ is the \textit{relative dualizing sheaf
$\omega_f=\omega_{X/Y}=\omega_{X/k}\otimes f^*\omega_{Y/k}^{\vee}$}. Let $f: X \to Y$ be a surface fibration of $g \ge 2$. We say that $f$ is \textit{relatively minimal} if
$X$ contains no $(-1)$-curves in fibers. In this situation, it is known that $\omega_f$
is nef.

\begin{theorem}[Relative Noether inequality] \label{relnoetheromega}
Let $f: X \to Y$ be a relatively minimal  fibration of genus $g \ge 2$ over $k$. Then
$$
h^0(\omega_f) \le \frac{g}{4g-4} \omega^2_f + g.
$$
\end{theorem}
The equality here can be obtained for certain nontrivial hyperelliptic fibration constructed in \cite{Xi2}.

The theorem can be viewed as a relative version of the classical Noether inequality on algebraic surfaces.
Recall that if $X$ is a minimal surface of general type over $k$, and
$\omega_X$ is the canonical bundle of $X$, then the
Noether inequality asserts that
$$
h^0(\omega_X) \le \frac 12 \omega^2_X +2.
$$
See \cite{BHPV} for $\charr k=0$ and \cite{Li1, Li2} for $\charr k>0$.
If the equality holds, then $X$ has a hyperelliptic pencil (cf.\cite{Ho}).

\subsection{Idea of the proof: a new filtration}
The idea of the proof of Theorem \ref{relnoether} is to construct a filtration. It comes from \cite{YZ}, where the arithmetic version
of this filtration is very natural in the setting of Arakelov geometry when studying the arithmetic linear series on arithmetic varieties.

Let $L$ be a nef line bundle on $X$.
We can find the smallest integer $e_L>0$ such that $L-e_LF$ is not nef. Then we take $L_1$ to be the movable part
of $|L-e_LF|$. Inductively, we get a filtration
$$
L = L_0 \ge L_1 \ge \cdots \ge L_n \ge 0.
$$
It corresponds to a family of decompositions:
$$
L=L_i+N_i+\sum_{j=0}^{i-1}e_{L_j}F
$$
Denote $L'_i=L_i-e_{L_i}F$ and one can easily compare the following two invariants:
$$
h^0(L'_i)-h^0(L'_{i+1}), \quad L'^2_i - L'^2_{i+1}.
$$
By taking the summation over $i$, we can prove Theorem \ref{relnoether}.

In \cite{Xi1}, Xiao introduced a filtration to study the geometry of fibered surface, namely, the Harder-Narasimhan
filtration. We would like to point out that our construction is essentially different from Xiao's method. In fact, Xiao's method
relies heavily on the (semi)stability of vector bundles on curves. Particularly, it is based on the fact that
\textit{the pull-back of a semistable vector bundle by a finite map between curves is still semistable}. This fact holds in characteristic zero, but fails in positive characteristics. For example, the Frobenius pullback can destabilize a (semi)stable bundle.
This is one main obstruction for Xiao's method to be available in full generality.
Another obstruction is that $f_*\omega_f$ is \textit{no longer} semipositive in positive characteristics. See \cite{MB} for example.

On the other hand, our filtration is constructed by some basic properties of linear series without any stability consideration. Hence it can be applied
regardless of the characteristic of the base field. It even works very well for arithmetic surfaces in \cite{YZ}, where the arithmetic versions of Theorem \ref{relnoether} and
\ref{relnoetheromega} are proved.

\subsection{New method towards the slope inequality}
An interesting consequence of our new filtration and relative Noether inequality is that, it gives a new outlook of the slope inequality. More precisely,
we can get a \textit{stability-free} proof of the slope inequality.

The slope inequality asserts that
$$
\omega^2_f \ge \frac{4g-4}{g} \deg (f_* \omega_f)
$$
for a relative minimal fibration of genus $g \ge 2$. In fact, all the known proofs are related to some sort of stability: geometric invariant theory in \cite{CH} (in characteristic $0$), stability of vector bundle in \cite{Xi1} (in charactistic $0$), Chow stability in \cite{Mo}.

Our treatment is completely different from any of the above. We find that in positive characteristics,
$$
({\rm relative \  Noether})  + ({\rm Frobenius \ iteration}) \Longrightarrow
({\rm slope \ inequality}).
$$
Using this idea, we only need to apply our relative Noether inequality to the Frobenius pull-back of the original fibration and take the limit.
The slope inequality in positive characteristics follows easily. In particular, if $g(Y) \le 1$, we can prove that the slope inequality even holds for fibrations whose generic fiber is not smooth. As a result, we prove the Arakelov inequality in positive characteristics, which improves
a result of \cite{LSZ}.

It turns out that this idea can even be applied to study the similar problem in characteristic $0$. In this case, the strategy can be expressed in short:
$$
({\rm slope \,\, ineq. \,\, in \ pos.\ char.})  +
({\rm reduction \,\, mod} \, \wp) \Longrightarrow
({\rm slope \, \,ineq. \,\, in\ char. \ zero}).
$$
Namely, in order to get the slope inequality in characteristic $0$, it is enough to prove it in positive characteristics.

It should be pointed out that in characteristic $0$, our filtration may not be as strong as the Harder-Narasimhan filtration in \cite{Xi1}. Nevertheless, when $g(Y) \le 1$, our method can recover the original slope inequality. When $g(Y) \ge 2$, our filtration might be slightly weaker. But at least for semistable fibrations, the relative Noether inequality works perfectly. Therefore, our method is definitely promising for studying the moduli space of curves.

\subsection{Geography problems}
Another related topic is the geography problem for irregular surfaces in positive characteristic.

The geography problem asks for the region of points in $\mathbb Z^2$ which are equal to $(c^2_1(X), c_2(X))$ for certain minimal surface $X$.
We refer to \cite{BHPV} for details over $\mathbb C$. However, this problem turns out to be very subtle in positive characteristics.
For example, $c_2$ can be negative for surfaces of general type in positive characteristics, and the positivity of $\chi$ is still unknown.
We refer to \cite{Li1, S-B} for details.

We study this problem in positive characteristics.
For example, we prove the Severi inequality in positive characteristics, i.e.,
$$
\omega^2_{X/k} \ge 4 \chi(\mathcal O_X)
$$
for smooth minimal surface $X$ of general type whose Albanese map is generically finite. In this situation, we know that
$\chi(\mathcal O_X)>0$ (c.f. \cite{Li1, S-B}).
When $k=\mathbb C$, it is proved by Pardini \cite{Pa} using a clever covering trick and the slope inequality of Xiao \cite{Xi1}.

It is worth mentioning that both the result and the method are new.
The behavior of the Albanese map in positive characteristics is not as good as in characteristic $0$.
See \cite{Li3} or \cite[Theorem 8.7]{Li1} for example, where the Albanese map is purely inseparable.
The covering trick in \cite{Pa} is still available, but due to the failure
of the Bertini's theorem for base-point-free line bundle in positive characteristic \cite{Jo}, the usual slope inequality for fibrations with smooth general fiber
is not enough for here. In particular, Moriwaki's result \cite{Mo} can not be applied in this case. Also, as we have mentioned before,
Xiao's method can not be applied here, either.

Nevertheless, our starting point is again different. We find that
$$
({\rm relative \,\, Noether}) + ({\rm covering \,\, trick}) \Longrightarrow
({\rm Severi \,\, inequality}).
$$
In this sense, the slope inequality is \textit{no longer} needed in the proof, even in characteristic $0$. In fact, the idea works uniformly in all characteristics.

It is also worth mentioning that our idea
works for arbitrary dimension. Just applying this idea, the higher-dimensional Severi inequality in characteristic $0$ is proved recently by the second author in \cite{Zh}. While, the slope inequality for higher dimensional fibration
is still far from being known.

Based on the same idea, we also study the case when the Albanese map of $X$ induces a fibration. We prove that
for smooth minimal surface $X$ of general type in positive characteristic with nontrivial Albanese map, the following are true:
\begin{itemize}
\item[(1)] $\omega^2_{X/k} \ge 2\chi(\mathcal O_X)$;
\item[(2)] If $2\chi(\mathcal O_X) \le \omega^2_{X/k} < 8 \chi(\mathcal O_X)/3$, then the Albanese map of $X$ induces a pencil
of genus $2$ curves.
\item[(3)] If $8 \chi(\mathcal O_X)/3 \le \omega^2_{X/k} < 3 \chi(\mathcal O_X)$, then the Albanese map of $X$ induces a pencil
of genus $2$ curves or hyperelliptic curves of genus $3$.
\end{itemize}
The corresponding results over $\mathbb C$ are proved by Bombieri \cite{Bom} for (1) and Horikawa \cite{Ho2} for (2) and (3).
In particular, in \cite{Ho2}, the slope inequality is essentially used.

Our proof of the result is again by the relative Noether formula. Simply speaking, we use
$$
({\rm relative \,\, Noether}) + ({\rm  base \,\, change})
$$
to prove these results. It also works in characteristic $0$.

\bigskip

{\bf Acknowledgement.}
The authors would like to thank Yanhong Yang for the communication explaining the subtlety of the stability in positive characteristic.
They are grateful to Xi Chen for a lot of valuable help and suggestions, and also Huayi Chen, Jun Lu, Sheng-Li Tan, Hang Xue, Kang Zuo for their interests on this paper. Special thanks goes to Christian Liedtke for his careful reading of the earlier version of this paper, his valuable comments and his unpublished preprint.

The first author is supported by the grant DMS-1161516 of the National Science Foundation.

\section{Proof of the relative Noether inequality}
In this section, we will prove Theorem \ref{relnoether} and \ref{relnoetheromega}.
The idea comes from \cite{YZ}, where a very similar argument was used to prove
an effective Hilbert-Samuel formula for arithmetic surfaces.


\subsection{The reduction process}
We first resume our notations.
Let $f: X \to Y$ be a  surface fibration, and $F$
be a general fiber. For any nef line bundle $L$ on $X$, we
can find a unique integer $c$ such that
\begin{itemize}
\item $L-cF$ is not nef;
\item $L-c'F$ is nef for any integer $c'<c$.
\end{itemize}
Denote this number $c$ for $L$ by $e_L$.
Since $L$ itself is nef, one has $e_L>0$.

We have the following lemma.
\begin{lemma} \label{onestep}
Let $f: X \to Y$ be a  surface fibration, and $F$
be a general fiber. Let $L$ be a nef line bundle on $X$ such that $LF>0$ and $h^0(L-e_LF) > 0$.
Then the linear system $|L-e_LF|$ has a fixed part $Z>0$. Moreover,
$$
ZF>0.
$$
\end{lemma}

\begin{proof}
By our assumption, we can write
$$
|L-e_LF|=|L_1| + Z,
$$
where $|L_1|$ is the movable part, and $Z$ is the fixed part. Since $L-e_LF$ is not nef, there
exists an irreducible and reduced curve $C$ on $X$ such that
$$
(L-e_LF)C<0,
$$
which implies $Z \ge C > 0$. Moreover, $LC \ge 0$ since $L$ is nef.
Therefore, $FC>0$ and
$ZF \ge FC > 0$.
\end{proof}

We have the following general theorem.
\begin{theorem}
Let $f: X \to Y$ be a  surface fibration, and $F$
be a general fiber. Let $L$ be a nef line bundle on $X$ such that $LF>0$ and $h^0(L) > 0$.
We can get the following sequence of triples
$$
\{(L_i, Z_i, a_i): i=0, 1, \cdots, n\}
$$
such that
\begin{itemize}
\item $(L_0, Z_0, a_0)=(L, 0, e_L)$ and $a_i=e_{L_i}$ for $i \ge 0$;
\item We have
$$
|L_{i-1}-e_{L_i}F|=|L_i|+Z_i,
$$
where $L_i$ (resp. $Z_i$) is the movable part (resp. fixed part) of the linear system $|L_{i-1}-e_{L_i}F|$ for $0<i<n$;
\item $h^0(L_n-a_nF)=0$;
\item $LF=L_0F > L_1F > \cdots > L_nF \ge 0$.
\end{itemize}

\end{theorem}

\begin{proof}
The triple $(L_{i+1}, Z_{i+1}, a_{i+1})$ can be obtained by applying Lemma \ref{onestep} to the triple
$(L_{i}, Z_{i}, a_{i})$. The whole process will terminate when $h^0(L_i-a_iF)=0$. It always terminates
because by Lemma \ref{onestep}, $L_iF$ decreases strictly.
\end{proof}

Now, we denote $r_i=h^0(L_i|_F)$, $d_i=L_iF$ and $L'_i=L_i-a_iF$.

\begin{prop} \label{algcase1}
For any $j=0, 1, \cdots n$, we have
\begin{eqnarray*}
h^0(L_0) & \le & h^0(L'_j) + \sum_{i=0}^j a_{i} r_{i}; \\
{L_0}^2 & \ge & 2a_0d_0 + \sum_{i=1}^j a_{i}(d_{i-1}+d_i) - 2d_0.
\end{eqnarray*}
\end{prop}

\begin{proof}
We have the short exact sequence:
$$
0 \longrightarrow H^0(L_{i+1}-F) \longrightarrow H^0(L_{i+1}) \longrightarrow H^0(L_{i+1}|_F).
$$
Then it follows that
$$
h^0(L_{i+1}-F) \le h^0(L_{i+1})-h^0(L_{i+1}|_F)=h^0(L_{i+1})-r_{i+1}.
$$
By induction, we have
$$
h^0(L'_{i+1})=h^0(L_{i+1}-a_{i+1}F) \le h^0(L_{i+1})-a_{i+1}r_{i+1} = h^0(L'_{i})-a_{i+1}r_{i+1}.
$$
Note that both $L'_i+F$ and $L'_{i+1}+F$ are nef, and $Z_i$ is effective. We get
\begin{eqnarray*}
{L'_{i}}^2-{L'^2_{i+1}} & = & (L'_i+L'_{i+1})(L'_i-L'_{i+1}) \\
&=& (L'_i+L'_{i+1})(a_{i+1}F+Z_{i+1}) \\
&=& a_{i+1}(L'_i+L'_{i+1})F + [(L'_i+F)+(L'_{i+1}+F)-2F]Z_{i+1}\\
& \ge & a_{i+1}(d_i+d_{i+1})-2Z_{i+1}F \\
& \ge & a_{i+1}(d_i+d_{i+1})-2(d_i-d_{i+1})
\end{eqnarray*}

For any $j=0, 1, \cdots, n-1$, summing over $i=0, 1, \cdots, j$, we have
\begin{eqnarray*}
h^0(L'_0) & \le & h^0(L'_j) + \sum_{i=1}^j a_i r_i; \\
L'^2_0 & \ge & {L'^2_{j}} + \sum_{i=1}^j a_{i}(d_{i-1}+d_i) - 2(d_0-d_j).
\end{eqnarray*}
Since we still have
\begin{eqnarray*}
L_0^2-{L'_0}^2 & = & 2a_0d_0, \\
L'^2_{j} + 2d_j & = & (L'_j + F)^2 \ge 0, \\
h^0(L_0) & \le & h^0(L'_0)+a_0r_0, \\
\end{eqnarray*}
the result follows.
\end{proof}

We also have the following lemma.

\begin{lemma}\label{algsumai}
In the above setting, we have
$$
2a_0+\sum_{i=1}^n a_i -2 \le \frac{L_0^2}{d_0}.
$$
\end{lemma}

\begin{proof}
Denote $b=a_1+\cdots+a_n$ and $Z=Z_1+ \cdots +Z_n$. We have the following linear equivalence
$$
L'_0 = L'_n + bF + Z.
$$
Since $L'_0+F$ and $L'_n+F$ are both nef, it follows that
\begin{eqnarray*}
(L'_0+F)^2& = &(L'_0+F)(L'_n+F+bF+Z) \\
& \ge & (L'_n+F+bF+Z)(L'_n+F) + bd_0 \\
& \ge & (L'_n+F)^2+b(d_0+d_n).
\end{eqnarray*}
Combine with
$$
L_0^2-(L'_0+F)^2=2(a_0-1)d_0.
$$
We get
$$
L_0^2 \ge (L'_n+F)^2+ 2(a_0-1)d_0 +b(d_0+d_n) \ge d_0(2a_0+b-2).
$$
\end{proof}

\subsection{Linear systems on curves}
We need several results on linear systems on algebraic curves.

Let $C$ be a reduced, irreducible Gorenstein curve over $k$. We say a line bundle $L$
is \textit{special} if
$$
h^0(L)>0, \qquad h^1(L)>0.
$$

We have the following theorem.

\begin{theorem} \label{h0bound}
Let $C$ be an irreducible, reduced Gorenstein curve over $k$, $p_a(C) \ge 2$. Let $L$ be a line bundle on $C$ such that $h^0(L)>0$ and $\deg(L) \le 2p_a(C)-2$.
\begin{itemize}
\item[(1)] {\rm [Clifford's Theorem]} If $L$ is special, then
$$
h^0(L) \le \frac 12 \deg(L) +1.
$$
Moreover, if $C$ is not hyperelliptic, then the equality holds if and only if $L=\mathcal{O}_C$ or $L=\omega_{C/k}$.
\item[(2)] If $h^1(L)=0$, then
$$
h^0(L) \le \frac 12 \deg(L).
$$
\end{itemize}
\end{theorem}

\begin{proof}
If $L$ is special, then the theorem is just the generalized version of Clifford's theorem in \cite{Li2}. If $h^1(L)=0$, by the Riemann-Roch theorem,
$$
h^0(L) = \deg(L) - p_a(C) + 1 \le \frac 12 \deg(L).
$$
\end{proof}

We also need the following lemma.
\begin{lemma} \label{hyper}
Let $L$ be a special line bundle on a hyperelliptic curve $C$ over $k$ such that $|L|$ is base-point-free, then $\deg(L)$
is even.
\end{lemma}

\begin{proof}
Denote $d_L=\deg(L)$. Since $L$ is base-point-free, we can choose $D_1, D_2 \in |L|$ and define a morphism
$$
\phi = (\phi_1, \phi_2) : C \to \mathbb P^1 \times \mathbb P^1.
$$
Here $\phi_1$ is the degree two map from $C$ to $\mathbb P^1$ by the definition of hyperelliptic curves, and
$\phi_2$ is the degree $d_L$ map $C$ to $\mathbb P^1$ induced by $D_1$ and $D_2$.
It is easy to see that either $\phi$ is birational from $C$ to $\phi(C)$ or $\deg \phi = 2$.

If $\deg \phi = 2$, then we are done. Furthermore, we claim that $\phi$ can not be birational to its image.

If $\phi$ is birational, denote $C'=\phi(C)$. By the definition of $\phi$, we can write
$C' \in |d_L F_1 + 2F_2|$, where $F_1$ and $F_2$ are rules on $\mathbb P^1 \times \mathbb P^1$.
Note that $L=\phi^* \mathcal O_{C'}(F_2)$. We have the following short exact sequence on $C'$:
$$
0 \longrightarrow \mathcal O_{C'} (F_2) \longrightarrow \phi_* L \longrightarrow \mathcal E \longrightarrow 0,
$$
where $\mathcal E$ is a skyscraper sheaf. Hence we have a surjection
$$
H^1(\mathcal O_{C'} (F_2)) \twoheadrightarrow H^1(\phi_* L).
$$
Since $L$ is special and $\phi$ is finite, we get $h^1(\mathcal O_{C'} (F_2))  \ge h^1(\phi_* L) = h^1(L) > 0$.

On the other hand, we have another exact sequence
$$
0 \longrightarrow \mathcal O_{\mathbb P^1 \times \mathbb P^1} (-C' + F_2) \longrightarrow \mathcal O_{\mathbb P^1 \times \mathbb P^1} (F_2)
\longrightarrow O_{C'} (F_2) \longrightarrow 0,
$$
which gives
$$
H^1(\mathcal O_{\mathbb P^1 \times \mathbb P^1} (F_2)) \longrightarrow H^1(\mathcal O_{C'} (F_2))
\longrightarrow H^2(O_{\mathbb P^1 \times \mathbb P^1} (-C'+ F_2)).
$$
Now by Serre duality,
$$
h^2(O_{\mathbb P^1 \times \mathbb P^1} (-C' + F_2)) = h^0(O_{\mathbb P^1 \times \mathbb P^1}((d_L-2)F_1-F_2)=0.
$$
Moreover, using the Riemann-Roch formula on $\mathbb P^1 \times \mathbb P^1$, we get
$$
h^1(\mathcal O_{\mathbb P^1 \times \mathbb P^1} (F_2))=0,
$$
which forces $h^1(\mathcal O_{C'} (F_2)) = 0$.
\end{proof}

We will divide the proof of Theorem \ref{relnoether} into two parts.

\subsection{Hyperelliptic case}

We first prove Theorem \ref{relnoether} when $F$ is hyperelliptic.

By Proposition \ref{algcase1} for $j=n$, we get
\begin{eqnarray*}
h^0(L_0) & \le & h^0(L'_n) + \sum_{i=0}^n a_{i} r_{i} = \sum_{i=0}^n a_{i} r_{i}, \\
L^2_0 & \ge & 2a_0d_0 + \sum_{i=1}^n a_{i}(d_{i-1}+d_i) - 2d_0.
\end{eqnarray*}

We need to deal with two different cases:
\begin{itemize}
\item Each $L_i|_F$ is special;
\item There exists a $k>0$ such that $L_0|_F, \cdots, L_{k-1}|_F$ are not special and $L_{k}|_F, \cdots, L_{n}|_F$ are special.
\end{itemize}

First, we assume that each $L_i|_F$ is special. If $d_0$ is even, applying Clifford's theorem, we have
$$
r_i \le \frac 12 d_i + 1.
$$
On the other hand, since $F$ is sufficiently general, and the linear system
$|L_i|$ has no fixed part for $i=1, \cdots, n$, the line bundle $L_i|_F$ is base-point-free
by construction. By Lemma \ref{hyper}, $d_i$ is even.
Thus
$$
d_{i-1}-d_i \ge 2
$$
for $i=1, \cdots, n$.
By Lemma \ref{algsumai}, we have
\begin{eqnarray*}
h^0(L_0) - \frac 14 L_0^2 & \le & \sum_{i=0}^n a_i - \frac 14 \sum_{i=1}^n(d_{i-1}-d_i) a_i + \frac 12 d_0  \\
& \le & a_0 + \frac 12 \sum_{i=1}^n a_i + \frac 12 d_0 \\
& \le & \frac {1}{2d_0} L_0^2 + \frac 12 d_0 + 1.
\end{eqnarray*}
If $d_0$ is odd, the only differences are
$$
r_0 \le \frac 12 d_0 + \frac 12,  \quad d_0-d_1 \ge 1.
$$
Therefore by Lemma \ref{algsumai} again,
\begin{eqnarray*}
h^0(L_0) - \frac 14 L_0^2 & \le & \sum_{i=0}^n a_i - \frac 14 \sum_{i=1}^n(d_{i-1}-d_i) a_i + \frac 12 d_0 - \frac 12 a_0 \\
& \le &  \frac 12 a_0 + \frac 34 \sum_{i=1}^n a_i + \frac 12 d_0 \\
& \le & \frac {3}{4d_0} L_0^2 + \frac 12 d_0 + \frac 12.
\end{eqnarray*}

Now let us deal with the other case. From Theorem \ref{h0bound}, we know
$$
r_i \le \frac 12 d_i, \qquad i=0, \cdots, k-1,
$$
and
$$
r_i \le \frac 12 d_i + 1, \qquad i=k, \cdots, n.
$$
Similarly, we have
$$
d_{i-1}-d_i \ge 1, \qquad i=0, \cdots, k,
$$
and
$$
d_{i-1}-d_i \ge 2, \qquad i=k+1, \cdots, n.
$$
Thus we have
\begin{eqnarray*}
h^0(L_0) - \frac 14 L_0^2
 &\le & \sum_{i=k+1}^n a_i -\frac 14 \sum_{i=k+1}^n(d_{i-1}-d_i) a_i + \frac 12 d_0 \\
 &&+ \left(r_k-\frac 12 d_k - \frac 14(d_{k-1}-d_k) \right) a_k \\
 & \le & \frac 12 \sum_{i=k+1}^n a_i + \frac 12 d_0 + \left(r_k-\frac 12 d_k - \frac 14(d_{k-1}-d_k) \right) a_k.
\end{eqnarray*}
If $d_{k-1}-d_k \ge 2$, we have
$$
h^0(L_0) - \frac 14 L_0^2 \le \frac 12 \sum_{i=k}^n a_i + \frac 12 d_0.
$$
If $d_{k-1}-d_k=1$, we know that $r_k \le r_{k-1}-1$. So
$$
r_k \le r_{k-1}-1 \le \frac 12 d_{k-1} - 1 = \frac 12 d_k + \frac 12.
$$
Hence we still have
$$
h^0(L_0) - \frac 14 L_0^2 < \frac 12 \sum_{i=k}^n a_i + \frac 12 d_0.
$$
By Lemma \ref{algsumai} again,
\begin{eqnarray*}
h^0(L_0) & \le & (\frac 14+ \frac{1}{2d_0}) L_0^2 + \frac 12 d_0 + 1 - a_0 - \frac 12 \sum_{i=1}^{k-1} a_i \\
& \le & (\frac 14+ \frac{1}{2d_0}) L_0^2 + \frac 12 d_0.
\end{eqnarray*}

\subsection{Non-hyperelliptic case}
In this case, for $i=1, \cdots, n-1$, we have the following stronger Clifford's theorem:
$$
r_i \le \frac12 d_i + \frac 12.
$$
For $i=0$ or $i=n$, it also holds if $L_i|_F$ is neither trivial nor $\omega_{F/k}$.

First, we assume $L_n|_F$ is not trivial. Using the strong bound and Lemma \ref{algsumai}, we get
\begin{eqnarray*}
h^0(L_0)-\frac 14 L^2_0 & \le & a_0 + \frac 12 \sum_{i=1}^n a_i - \frac 14 \sum_{i=1}^n a_i(d_{i-1}-d_i)  + \frac 12 d_0 \\
& \le & a_0 + \frac 12 \sum_{i=1}^n a_i + \frac 12 d_0 \\
& \le & \frac {1}{2d_0} L^2_0 + \frac 12 d_0 + 1.
\end{eqnarray*}
If $L_n|_F$ is trivial, then $r_n=1$ and $d_n=0$. Note that $d_{n-1}-d_n > 2$ since $F$ is not hyperelliptic. It gives
\begin{eqnarray*}
h^0(L_0)-\frac 14 L^2_0 & \le & a_0 + \frac 12 \sum_{i=1}^{n-1} a_i + a_n - \frac 14 \sum_{i=1}^{n} a_i (d_{i-1}-d_i) + \frac 12 d_0 \\
& < & a_0 + \frac 12 \sum_{i=1}^{n} a_i + \frac 12 d_0.
\end{eqnarray*}
By Lemma \ref{algsumai} again, it yields
$$
h^0(L_0) \le (\frac 14 + \frac {1}{2d_0}) L^2_0 + \frac 12 d_0 + 1.
$$
It ends the proof of Theorem \ref{relnoether}.

\subsection{Proof of Theorem \ref{relnoetheromega}}

Now, Theorem \ref{relnoetheromega} becomes straightforward.
Let $f: X \rightarrow Y$ be a relatively minimal fibration. Then the relative dualizing sheaf $\omega_f$ is nef and $\omega_f F= 2g-2$ is even.
Just applying Theorem \ref{relnoether}, we can directly get Theorem \ref{relnoetheromega}.


\section{Slope inequality}

In this section, we prove the slope inequality.
By the Riemann-Roch theorem, Theorem \ref{relnoetheromega} implies the slope inequality with an ``error term''.
To get rid of the ``error term'', we consider certain base change of $Y$ and take a limit.
The argument is straightforward if $g(Y) \le 1$.
In general, the covering trick only works in positive characteristics.
To get the result in characteristic $0$, we use a reduction argument.

\subsection{Slope inequality for $g(Y) \le 1$}

In this section, we apply Theorem \ref{relnoetheromega} to give a new proof of the slope inequality for
fibered surfaces over $\mathbb P^1$ and elliptic curves.
We first give the following lemma.

\begin{lemma}\label{h0todeg}
Let $f: X \rightarrow Y$ be a relatively minimal  surface fibration of genus $g$. Then
$$
\deg(f_* \omega_f) \le \frac {g}{4g-4} \omega^2_f  + gb.
$$
Here $b=g(Y)$.
\end{lemma}

\begin{proof}
By Theorem \ref{relnoetheromega}, we have
$$
h^0(\omega_f) \le \frac{g}{4g-4} \omega^2_f + g.
$$
On the other hand, using the Riemann-Roch theorem on $Y$,
$$
h^0(\omega_f)=h^0(f_*\omega_f) \ge \deg(f_*\omega_f)+g(1-b).
$$
It follows that
$$
\deg(f_*\omega_f) \le \frac{g}{4g-4} \omega^2_f + gb.
$$
\end{proof}

\begin{theorem} \label{slopeg1}
Let $f: X \rightarrow Y$ be a relatively minimal  surface fibration of genus $g$. Assume that
$g(Y) \le 1$. Then
$$
\omega^2_f \ge \frac {4g-4}{g} \deg(f_* \omega_f).
$$
\end{theorem}

\begin{proof}
If $g(Y)=0$, then the result is just Lemma \ref{h0todeg}.

Now, suppose that $Y$ is an elliptic curve over $k$ and that $\mu: Y \to Y$ is the multiplication by $n$ such that
$n$ and $\charr k$ are coprime with each other.
Denote $X'=X \times_{\mu} Y$.
We get a new fibration $f' : X' \rightarrow Y$ which is just the pull-back of $f$ by $\mu$. Applying
Lemma \ref{h0todeg} to $f'$, it follows that
$$
\deg(f'_*\omega_{f'}) \le \frac{g}{4g-4} \omega^2_{f'} + g.
$$
On the other hand, since the base change is \'etale, we have the following facts:
$$
\deg(f'_*\omega_{f'})= n^2 \deg(f_*\omega_f), \quad \omega^2_{f'}=n^2 \omega^2_{f},
$$
which gives us
$$
n^2 \deg(f_*\omega_f) \le \frac{n^2 g}{4g-4} \omega^2_f + g,
$$
i.e.,
$$
\deg(f_*\omega_f) \le \frac{g}{4g-4} \omega^2_f + \frac g {n^2}.
$$
We can prove our result by letting $n \rightarrow \infty$.
\end{proof}

\begin{remark}
In the case $b=g(Y) \ge 2$, if we directly use an \'etale base change $\pi : Y' \rightarrow Y$, then
$g(Y')$ also increases. Therefore, we can not prove the general slope inequality using the above argument.
However, we still want to use the base change
trick and control $g(Y')$ at the same time, which motivates us to consider the reduction mod $\wp$ method.
\end{remark}

\subsection{Slope inequality in positive characteristic}
We first prove the slope inequality when $\charr k=p>0$. This is also crucial for us to prove
the slope inequality for $\charr k=0$. Actually, the proof is quite similar to the proof of Theorem \ref{slopeg1}.

\begin{theorem} \label{slopeposchar}
Let $f: X \to Y$ be a semistable fibration of genus $g$ over $k$ of positive characteristic, and $g(Y) \ge 2$. Then the slope inequality holds.
\end{theorem}

\begin{remark}
By a result of Tate \cite{Ta}, the general fiber is smooth under this assumption.
\end{remark}

\begin{proof}
Let $f: X \to Y$ be a semistable  fibration of
genus $g$. By
Lemma \ref{h0todeg} to $f$, it follows that
$$
\deg(f_*\omega_{f}) \le \frac{g}{4g-4} \omega^2_{f} + gb,
$$
where $b=g(Y)$.

Now let $F_Y : Y \to Y$ be the absolute Frobenius morphism of $Y$. Via this base change, we get a new  fibration
$$
f' : X' \to Y,
$$
where $X'$ is the minimal desingularization of the normal surface $X \times_{F_Y} Y$. Thus $f'$ is still
semistable.
Applying Lemma \ref{h0todeg} again to $f'$, it follows that
$$
\deg(f'_*\omega_{f'}) \le \frac{g}{4g-4} \omega^2_{f'} + gb.
$$
Moreover, we
have the following facts:
$$
\deg(f'_*\omega_{f'})= p \deg(f_*\omega_f), \quad \omega^2_{f'}=p \, \omega^2_{f},
$$
We obtain
$$
p\deg(f_*\omega_f) \le \frac{pg}{4g-4} \omega^2_f + gb,
$$
i.e.,
$$
\deg(f_*\omega_f) \le \frac{g}{4g-4} \omega^2_f + \frac {gb}{p}.
$$
We can prove our result by iterating this Frobenius base change.

\end{proof}

Actually, using the same idea, we can get a more general result for line bundles of small degree in positive
characteristics, which is similar to Theorem \ref{slopeposchar}.

\begin{theorem} \label{relspecialslope}
Let $f: X \to Y$ be a  surface fibration of genus $g>0$ over a field $k$ of positive characteristic,
and $L$ be a nef line bundle on $X$. Assume that $2 \le d=\deg(L|_F) \le 2g-2$, where $F$ is a general fiber
of $f$. Then
$$
L^2 \ge \frac{4d}{d+2+\varepsilon} \deg(f_* L).
$$
Here, $\varepsilon=1$ if $F$ is hyperelliptic, $d$ is odd and $L|_F \le K_F$. Otherwise, $\varepsilon=0$.
\end{theorem}

\begin{proof}
The proof is the same as the proof of the slope inequality in positive characteristic. We just sketch here.

First, by Theorem \ref{relnoether}, we have
$$
h^0(L) \le (\frac 14 + \frac {2+\varepsilon}{4d}) L^2 + \frac {d+2-\varepsilon}{2}.
$$
The Riemann-Roch theorem on $Y$ gives us
$$
h^0(L)=h^0(f_* L) \ge \deg (f_* L) + r(1-b),
$$
where $b=g(Y)$ and $r=h^0(L|_F)$. Combine them together and we get
$$
\deg(f_*L) \le (\frac 14 + \frac {2+\varepsilon}{4d}) L^2 + \frac {d+2-\varepsilon}{2} + r(b-1).
$$
Now we apply the Frobenius base change iteration as above. Finally, we eliminate the constant term.
\end{proof}

\subsection{Slope inequality in characteristic zero}
Now we can prove the slope inequality for $\charr k=0$.
We have the following lemma.

\begin{lemma} \label{redmodp}
Let $X, Y, Z$ be integral schemes, and $f: X\to Y$ and $g:Y \to Z$ be proper and flat morphisms of relative dimension one.
Assume that $f$ is a local complete intersection and $g$ is smooth. Then the numbers
$$
\deg {f_{z}}_*(\omega_{X_z/Y_z}) \quad \mbox{and} \quad \omega^2_{X_z/{Y_z}}
$$
are independent of $z\in Z$.
Here $f_z:X_z\to Y_z$ denotes the fiber of $f:X\to Y$ over $z$,  and
$\omega_{X_z/Y_z}$ denotes the relative dualizing sheaf of $X_z$ over $Y_z$.

\end{lemma}

\begin{proof}
The invariance of $\deg {f_{z}}_*(\omega_{X_z/Y_z})$ is an interpretation of the determinant line bundle.
Recall that for any line bundle $L$ on $X$, the determinant line bundle
$\lambda_f(L)$ is a line bundle on $Y$ such that, for any $y\in Y$, there is a canonical isomorphism
$$
\lambda_f(L) \simeq \det H^*(X_y,L_y)
=\det H^0(X_y,L_y) \otimes  \det H^1(X_y,L_y)^\vee.
$$
The construction is functorial.

In the setting of the lemma, consider the determinant line bundle
$M=\lambda_f(\omega_{X/Y})$.
Restricted to $Y_z$, we have
$$M|_{Y_z}=(\det f_{z*}\omega_{X_z/Y_z}) \otimes  (\det R^1f_{z*}\omega_{X_z/Y_z} )^\vee
=\det f_{z*}\omega_{X_z/Y_z}. $$
Here we used the canonical isomorphism $R^1f_{z*}\omega_{X_z/Y_z} = \mathcal O_{Y_z}$ following the duality theorem.
Therefore,
we simply have
$$\deg f_{z*}\omega_{X_z/Y_z}
=\deg (M|_{Y_z}). $$
It is independent of $z$.

The invariance of $\omega^2_{X_z/{Y_z}}$ follows from the definition of the Deligne pairing introduced in \cite{De}.
In fact, the Deligne pairing $N=\langle \omega_{X/Y},\omega_{X/Y} \rangle$
is a line bundle on $Y$ such that
$$
\omega^2_{X_z/{Y_z}}=\deg (N|_{Y_z}), \quad\forall z\in Z
$$
It is also independent of $z$.
\end{proof}

The proof does not work for arbitrary line bundles of small degree since we used the duality theorem above.

Go back to our setting. Suppose that $f : X \to Y$ is a semistable fibration over of characteristic $0$. By the Lefschetz principle, one can assume that $k$ is finitely generated over $\mathbb Q$. Let $\mathcal Z$ be an integral scheme of finite type over $\mathbb Z$ with the function field $k$. Shrink $\mathcal Z$ by an open subset and replace it by a finite cover if necessary. We are
able to extend the composition $X \to Y\to \mathrm{Spec}(k)$ to $\mathcal X \to \mathcal Y \to \mathcal Z$ satisfying the conditions of the lemma.

Now choose a nonzero prime $\wp \in \mathcal Z$ such that $f_{\wp} : \mathcal X_{\wp} \to \mathcal Y_{\wp}$ is a semistable fibration over the field $k/\wp$. By the slope inequality in positive
characteristic, we have
$$
\omega^2_{f_{\wp}} \ge \frac {4g-4}{g} \deg({f_\wp}_* \omega_{f_{\wp}}).
$$
By Lemma \ref{redmodp}, the slope inequality holds for $f$ over $k$.



\subsection{Arakelov inequality}
As a direct consequence of the slope inequality, we have the following Arakelov inequality in positive characteristic.
\begin{theorem} [Arakelov inequality] \label{arakelov}
Let $f$ be a non-isotrivial semistable fibration of genus $g \ge 2$ over the field $k$ of positive characteristic. Let $S \subset Y$ be the singular locus over which $f$ degenerates. Then
$$
\deg(f_* \omega_f) < g^2 \deg \Omega_Y(S).
$$
\end{theorem}
\begin{proof}
Let $f$ be a non-isotrivial semistable fibration of genus $g \ge 2$ over $k$. Let $S \subset Y$ be the singular locus over which $f$ degenerates.
We have the following Szpiro's inequality (cf. \cite[Prop. 4.2]{Sz}):
$$
\omega^2_f < (4g-4) g \deg \Omega_{Y}(S).
$$
Therefore,
$$
\deg(f_* \omega_f) \le \frac{g}{4g-4} \omega^2_f < g^2 \deg \Omega_{Y}(S).
$$
\end{proof}

\section{Irregular surfaces in positive characteristic}
In this section, we always assume that $X$ is a smooth minimal surface of general type
over a field $k$ of positive characteristic.

\subsection{Proof of Severi inequality}
We have the following theorem.
\begin{theorem}\label{severi}
Suppose the Albanese map of $X$ is generically finite. Then
$$
\omega^2_{X} \ge 4\chi(\mathcal O_X).
$$
\end{theorem}

Before the proof, we first show a lemma.

\begin{lemma} \label{easylem}
Let $X$ be a smooth surface of general type over $k$. Let $B$ be a nef and big line bundle on $X$ such that
$|B|$ is base point free. Then for any nef line bundle $L \le K_X+B$ satisfying $d=LB \ge 2$, we have
$$
h^0(L) \le (\frac 14 + \frac 1 {2d}) L^2 + \frac {d+2}{2}.
$$
\end{lemma}

\begin{proof}
By Bertini's theorem in \cite{Jo}, we can choose two general member $B_1, B_2 \in |B|$ such that $B_1, B_2$ are both geometrically integral
and intersect each other properly. Now, let $\sigma: \widetilde X \to X$ be the blow-up of $X$ along $B_1 \cap B_2$ and we get a fibration from $\widetilde X$
to $\mathbb P^1$ with the general fiber $\widetilde B$ the proper transformation of $B_i$. Apply Theorem \ref{relnoether} to this fibration. It follows that
$$
h^0(L)=h^0(\sigma^*L) \le (\frac 14 + \frac 1 {2d}) (\sigma^*L)^2 + \frac {d+2}{2} = (\frac 14 + \frac 1 {2d}) L^2 + \frac {d+2}{2}.
$$
\end{proof}

\begin{proof}[Proof of Theorem \ref{severi}]
Let $X$ be a minimal surface of general type over $k$ with maximal Albanese dimension. Denote
${\rm Alb}_X: X \to A$ to be the Albanese map, where $A={\rm Alb}(X)$ is an abelian variety of dimension
$m$ over $k$.
We only have
$m \le h^{1,0}(X)$ (cf. \cite{Ig}),
but since $X$ is of maximal Albanese dimension, it is always safe for us to use the bound
$$
m \ge 2.
$$

Let $H'$ be a very ample line bundle on $A$, and $L$ be the pull-back of $H=2H'$
on $X$. Set
$$
\alpha = L^2, \quad \beta = \omega_{X/k} L.
$$
Since $X$ is of general type, $\alpha$ and $\beta$ are both strictly positive.

Let $\mu: A \to A$ be the multiplication by $n$, where $n>1$ is an integer. If $\charr k=p>0$, we assume
$n$ and $p$ to be coprime.
We have the following
base change:
$$
\xymatrix{X' \ar[r]^{\tau} \ar[d]_{\nu} & X \ar[d]^{{\rm Alb}_X} \\
A \ar[r]^\mu & A}
$$
where $X'=X \times_\mu A$. We have
$$
\omega^2_{X'/k} = n^{2m} \omega^2_{X/k}, \quad \chi(\mathcal O_{X'})=n^{2m}\chi(\mathcal O_{X}).
$$
We also have the following numerically equivalence on $A$:
$$
\mu^*H \sim_{\rm num} n^2 H,
$$
which yields
$$
\tau^* L \sim_{\rm num} n^2 L'.
$$
Here $L'=\nu^* H$. It follows that
$$
L'^2=n^{2m-4} \alpha, \quad \omega_{X'_n/k}L'=n^{2m-2}\beta.
$$
Now, apply Lemma \ref{easylem} to $X'$. It follows that
$$
\chi(\mathcal O_{X'}) \le h^0(\omega_{X'_n/k}) \le (\frac 14 + \frac 1{2n^{2m-2}\beta}) \, \omega^2_{X'_n/k} + \frac{n^{2m-2}\beta + 2}{2}.
$$
Hence one has
$$
(\frac 14 + \frac 1{2n^{2m-2}\beta}) \, n^{2m} \omega^2_{X/k} + \frac{n^{2m-2}\beta + 2}{2} \ge n^{2m}\chi(\mathcal O_{X}).
$$
Therefore, the results follows by letting $n \to \infty$.
\end{proof}

\subsection{Surfaces with Albanese pencils}

Here we treat the following more general case.
\begin{theorem}
Suppose $\chi(\mathcal O_X) > 0$ and the Albanese map of $X$ is non-constant. Then the following are true:
\begin{itemize}
\item[(1)] $\omega^2_{X} \ge 2\chi(\mathcal O_X)$;
\item[(2)] If $2\chi(\mathcal O_X) \le \omega^2_{X} < 8 \chi(\mathcal O_X)/3$, then the Albanese map of $X$ induces a pencil
of genus $2$ curves.
\item[(3)] If $8 \chi(\mathcal O_X)/3 \le \omega^2_{X} < 3 \chi(\mathcal O_X)$, then the Albanese map of $X$ induces a pencil
of genus $2$ curves or hyperelliptic curves of genus $3$.
\end{itemize}
\end{theorem}

\begin{proof}
We can assume  $\omega^2_{X} < 3 \chi(\mathcal O_X)$. Then from Theorem \ref{severi}, we know that the Albanese map of $X$ induces a pencil.
Passing through the Stein factorization,
we denote this pencil by $f: X \to Y$, where $Y$ is a smooth curve satisfying $b=g(Y)>0$, the general fiber $F$ of $f$ is geometrically integral of arithmetic genus $g=p_a(F) \ge 2$.

Let $A={\rm Jac}(Y)$ and $\mu: A \to A$ be the multiplication by a positive integer $n$ coprime to $\charr k$. Let $X'=X \times_{\mu} A$. We have the following diagram:
$$
\xymatrix{X' \ar[r]^{\tau} \ar[d]_{\nu} & X \ar[d]^{f} \\
Y' \ar@{->}[d] \ar[r]^{\mu} & Y \ar[d] \\
A \ar[r]^{\mu} & A}
$$

Apply Theorem \ref{relnoether} to $\omega_{X'}$. One can get
$$
\chi(\mathcal O_{X'}) \le h^0(\omega_{X'}) \le \frac{g}{4g-4} \omega^2_{X'} + g.
$$
It yields
$$
n^{2b} \chi(\mathcal O_X) \le \frac{n^{2b}g}{4g-4} \omega^2_{X}+g.
$$
By letting $n \to \infty$, it follows that
$$
\omega^2_{X} \ge \frac{4g-4}{g} \chi(\mathcal O_X).
$$
Therefore, from our assumptions, (1) and (2) are proved and $g \le 3$. We devote the rest part to proving (3).

Fix a very ample line bundle $H$ on $A$ and write $L'=\nu^*H$. Then
$$
h^0(L') = h^0(H|_{Y'}) \le \deg(H|_{Y'}) \le n^{2b-2} \deg(H|_Y).
$$
While we have seen that $h^0(\omega_X) \ge n^{2b} \chi(\mathcal O_X)$. Then we can assume that
$$
\omega_{X'}-L' \ge 0
$$
for $n$ sufficiently large.
It implies that $\nu$ factors through the canonical map
$\phi_{\omega_{X'}}$ of $X'$. We can further assume that
$$
\omega^2_{X'} < 3\chi(\mathcal O_{X'}) - 10.
$$

We claim that  $\nu$ is the only possible hyperelliptic pencil on $X'$.
Otherwise, $X'$ has a hyperelliptic pencil which is not $\nu$. Denote the general member in this pencil by $G$. Then $\nu(G)=Y'$.
On the other hand, we know $\phi_{\omega_{X'}}(G)=\mathbb P^1$. Hence $g(Y')=0$, which is impossible. Therefore, our claim holds.

Now suppose that $F$ is not hyperelliptic.
If $\phi_{\omega_{X'}}$ is not generically finite, then we can write
$$
\omega_{X'} \sim_{\rm alg} aF'+Z,
$$
where $F'$ is a general fiber of $\nu$, $Z$ is the fixed part of $|\omega_{X'}|$, $a \ge h^0(\omega_{X'})$. Hence
$$
\omega^2_{X'} \ge a \omega_{X'} F' \ge 4 h^0(\omega_{X'}) \ge 4 \chi(\mathcal O_{X'}) - 4.
$$
If $\phi_{\omega_X}$ is not generically finite, by the Castelnuovo inequality in \cite{Li1},
$$
\omega^2_{X'} \ge 3h^0(\omega_{X'}) - 7 \ge  3 \chi(\mathcal O_{X'})-10.
$$
Neither case is possible. So we finish this proof.
\end{proof}

\end{document}